\documentclass[11pt]{article}
\usepackage{soul}
\usepackage{exscale,relsize}
\usepackage{enumerate}
\usepackage{amsmath}
\usepackage{amsfonts}
\usepackage{amssymb}
\usepackage{calc}
\usepackage{theorem}
\usepackage{pifont}      
\usepackage{xcolor}
\newcommand{\infconv}{\ensuremath{\mbox{\small$\,\square\,$}}}

\newcommand{\scal}[2]{\langle{{#1},{#2}}\rangle}

\newcommand{\RR}{\ensuremath{\mathbb R}}

\newcommand{\RX}{\ensuremath{\,\left]-\infty,+\infty\right]}}

\newcommand{\thalb}{\ensuremath{\tfrac{1}{2}}}

\newcommand{\zu}{\ensuremath{{\left]0,1\right[}}}

\newcommand{\moro}[2]{\ensuremath{e_{#1}\,{#2}}}
\newcommand{\goe}[2]{\ensuremath{G_{#1}\,{#2}}}
\newcommand{\shawn}[2]{\ensuremath{S_{#1}\,{#2}}}
\newcommand{\pav}[4]{\ensuremath{\operatorname{pav}\left({#1},{#2}\,;{#3},{#4}\right)}}

\newcommand{\Id}{\ensuremath{\operatorname{Id}}}

\newcommand{\pinf}{\ensuremath{+\infty}}

\renewcommand{\phi}{\ensuremath{\varphi}}

\newcommand{\qq}{\ensuremath{\mathfrak{q}}}

\newtheorem{theorem}{Theorem}[section]

\newtheorem{definition}[theorem]{Definition}

\theoremstyle{plain}{\theorembodyfont{\rmfamily}
}
\theoremstyle{plain}{\theorembodyfont{\rmfamily}
}
\theoremstyle{plain}{\theorembodyfont{\rmfamily}
}
\theoremstyle{plain}{\theorembodyfont{\rmfamily}
\newtheorem{example}[theorem]{Example}}
\theoremstyle{plain}{\theorembodyfont{\rmfamily}
\newtheorem{remark}[theorem]{Remark}}
\theoremstyle{plain}{\theorembodyfont{\rmfamily}
}

\begin{document}

\title{{\sffamily Self-dual Smooth Approximations of Convex Functions\\
via the Proximal Average}}

\author{
Heinz H.\ Bauschke\thanks{Mathematics, Irving K.\ Barber School,
UBC Okanagan, Kelowna, British Columbia V1V 1V7, Canada. E-mail:
\texttt{heinz.bauschke@ubc.ca}.}, ~ Sarah M.\
Moffat\thanks{Mathematics, Irving K.\ Barber School, UBC Okanagan,
Kelowna, British Columbia V1V 1V7, Canada.
E-mail:  \texttt{smoffat99@gmail.com}.},~ and ~  Xianfu
Wang\thanks{Mathematics, Irving K.\ Barber School, UBC Okanagan,
Kelowna, British Columbia V1V 1V7, Canada. E-mail:
\texttt{shawn.wang@ubc.ca}.} }

\date{
March 30, 2010 
}

\maketitle


\begin{abstract} \noindent
The proximal average of two convex functions has proven to be a useful tool
in convex analysis. In this note, we express Goebel's
self-dual smoothing operator in terms of the proximal average, which allows
us to give a simple proof of self duality. 
We also provide a novel self-dual smoothing operator.
Both operators are illustrated by smoothing the norm. 
\end{abstract}



\section{Introduction}

Let $X$ be the standard Euclidean space $\RR^n$, 
with inner product $\scal{\cdot}{\cdot}$ and induced norm $\|\cdot\|$.
It will be convenient to set
\begin{equation}
\qq = \thalb\|\cdot\|^2.
\end{equation}
Now let $f\colon X\to\RX$ be convex, lower semicontinuous, and proper.
Since many convex functions are nonsmooth, it is natural to ask:
How can one approximate $f$ with a smooth function?

The most famous and very useful answer to this question is provided
by the \emph{Moreau envelope} \cite{Moreau,RockWets},
which, for $\lambda>0$, is defined by\,\footnote{The symbol ``$\infconv$''
denotes \emph{infimal convolution}: $(f_1\infconv f_2)(x) = \inf_{y}
\big(f_1(y) + f_2(x-y)\big)$.}
\begin{equation}
\moro{\lambda}{f} = f\infconv \lambda^{-1}\qq.
\end{equation}
It is well known that $\moro{\lambda}{f}$ is smooth and that
$\lim_{\lambda\to 0^+}\moro{\lambda}{f}=f$ point-wise;
see, e.g., \cite[Theorem~1.25 and Theorem~2.26]{RockWets}. 

Let us consider the norm,
which is nonsmooth at the origin. 

\begin{example}[Moreau envelope of the norm]
\label{ex:moronorm}
Let $\lambda\in\zu$, set $f=\|\cdot\|$, and denote
the closed unit ball by $C$. 
Then, for $x$ and $x^*$ in $X$, we have\,\footnote{
Here $\iota_{C}(x)=0$, if $x\in C$; $\iota_C(x)=\pinf$, if $x\notin C$
is the \emph{indicator function}, 
$f^*(x^*) = \sup_{x\in X}\big(\scal{x}{x^*}-f(x)\big)$ is
the \emph{Fenchel conjugate} of $f$, and
$d_C = \|\cdot\|\infconv \iota_C$ is the \emph{distance function}.}
\begin{equation}
\label{e:moronorm}
\moro{\lambda}{f}(x) = 
\begin{cases}
\displaystyle \frac{\|x\|^2}{2\lambda}, &\text{if $\|x\|\leq
\lambda$;}\\[+4mm]
\displaystyle \|x\|-\frac{\lambda}{2}, &\text{if $\|x\|>\lambda$,}
\end{cases}
\end{equation}
$\big(\moro{\lambda}{f}\big)^* = \iota_C + \lambda\qq$, and
$\moro{\lambda}{(f^*)}(x^*) = 
(2\lambda)^{-1}\cdot\big(\max\{0,\|x^*\|-1\}\big)^2$. 
Consequently, 
$\big(\moro{\lambda}{f}\big)^* \neq \moro{\lambda}{(f^*)}$. 
\end{example}
\begin{proof}
Either a straight-forward computation
or \cite[Example~11.26(a)]{RockWets} yields
\begin{equation}
f^* = \iota_C.
\end{equation}
Next, if $y\in X$, then
\begin{align}
\moro{1/\lambda}{\iota_C}(y) &= \inf_{c\in C} 
\lambda\qq(y-c)\\[+4mm]
&= \frac{\lambda}{2} d_C^2(y)\\[+4mm]
&=\frac{\lambda}{2}\cdot
\begin{cases}
\big(\|y\|-1\big)^2, &\text{if
$\|y\|>1$;}\\[+4mm]
0, &\text{if $\|y\|\leq 1$,}
\end{cases}
\end{align}
and thus
\begin{equation}
\moro{1/\lambda}{\iota_C}\big(x/\lambda\big)  = 
\frac{\lambda}{2}\cdot\begin{cases}
\big(\|x/\lambda\|-1\big)^2, &\text{if
$\|x\|>\lambda$;}\\[+4mm]
0, &\text{if $\|x\|\leq \lambda$.}
\end{cases}
\end{equation}
By \cite[Example~11.26(b) on page~495]{RockWets}, we obtain
\begin{align}
\moro{\lambda}{f}(x) &= \frac{1}{\lambda}\qq(x) -
\moro{1/\lambda}{f^*}\big(x/\lambda\big)\\
&=\frac{1}{2\lambda}\|x\|^2 - \frac{\lambda}{2}\cdot
\begin{cases}
\displaystyle \frac{\|x\|^2}{\lambda^2} - \frac{2\|x\|}{\lambda} + 1,
&\text{if $\|x\|>\lambda$;}\\[+4mm]
0, &\text{if $\|x\|\leq \lambda$}
\end{cases}\\[+4mm]
&= \begin{cases}
\displaystyle \|x\|-\frac{\lambda}{2}, &\text{if $\|x\|>\lambda$;}\\[+4mm]
\displaystyle\frac{\|x\|^2}{2\lambda}, &\text{if $\|x\|\leq \lambda$}
\end{cases}
\end{align}
and $\big(\moro{\lambda}{f}\big)^* = f^* + \lambda\qq = 
\iota_C + \lambda\qq$.
Alternatively, one may use \cite[Example~2.16]{CW}, which provides the
proximal mapping of $f$, 
and then use the proximal mapping calculus to obtain these results. 
\end{proof}

While the Moreau envelope has many desirable properties, we see from
Example~\ref{ex:moronorm} that the smooth approximation $\moro{\lambda}{f}$
is not \emph{self-dual} in the sense that
\begin{equation}
\big(\moro{\lambda}{f}\big)^* \neq \moro{\lambda}{(f^*)}.
\end{equation}

It is perhaps surprising that self-dual smoothing operators even exist.
The first example appears in \cite{Rafal}. Specifically, Goebel
defined
\begin{equation}
\goe{\lambda}{f} = (1-\lambda^2)\moro{\lambda}{f} + \lambda\qq
\end{equation}
and proved that
\begin{equation}
\big(\goe{\lambda}{f}\big)^* = \goe{\lambda}{(f^*)},
\end{equation}
i.e., \emph{Fenchel conjugation and Goebel smoothing commute!}
For applications of his smoothing operator, see \cite{Rafal}. 

\noindent
\emph{The purpose of this note is two-fold.
First, we present a different representation of the Goebel smoothing
operator which allows us to prove self-duality using the Fenchel
conjugation formula for the proximal average. 
Secondly, the proximal average is also
utilized to obtain a novel smoothing operator.
Both smoothing operators are computed explicitly for the norm.
The formulas derived show that the new smoothing operator is distinct
from the one provided by Goebel.
}

For $f_1$ and $f_2$, two functions from $X$ to $\RX$ that 
are convex, lower semicontinuous and proper, 
and for two strictly positive convex coefficients
($\lambda_1+\lambda_2=1$), 
the \emph{proximal average} is defined by 
\begin{equation}
\label{e:pavdef}
\pav{f_1}{f_2}{\lambda_1}{\lambda_2} =  \big(\lambda_1(f_1+\qq)^* +
\lambda_2(f_2+\qq)^*\big)^* - \qq.
\end{equation}
See \cite{PA,BLT,BLW,BMR,BW,Rafal,Rafal2} 
for further information and applications of the proximal average.
The key property
is the \emph{Fenchel conjugation formula}
\begin{equation}
\label{e:pavconj}
\pav{f_1}{f_2}{\lambda_1}{\lambda_2}^* = 
\pav{f_1^*}{f_2^*}{\lambda_1}{\lambda_2};
\end{equation}
see \cite[Theorem~6.1]{BMR}, \cite[Theorem~4.3]{BLT}, or
\cite[Theorem~5.1]{PA}. 

We use standard convex analysis calculus and notation as, 
e.g., in \cite{Rocky,RockWets,Zali}. 
In Section~\ref{sec:Rafal}, we consider Goebel's smoothing operator
from the proximal-average view point.
The new smoothing operator is presented in Section~\ref{sec:Shawn}.

\section{The Goebel smoothing operator}

\label{sec:Rafal}

\begin{definition}[Goebel smoothing operator]
Let $f\colon X\to\RX$ be convex, lower semicontinuous and proper,
and let $\lambda \in\zu$.
Then the \emph{Goebel smoothing operator \cite{Rafal}} is defined by
\begin{equation}
\label{e:raf}
\goe{\lambda}{f} = (1-\lambda^2)\moro{\lambda}{f} + \lambda\qq.
\end{equation}
\end{definition}

Note that \eqref{e:raf} and standard properties of the Moreau envelope 
imply that point-wise
\begin{equation}
\lim_{\lambda\to 0^+} \goe{\lambda}{f} = f
\end{equation} and
that each $\goe{\lambda}{f}$ is smooth.

Our first main result provides two alternative descriptions of the Goebel
smoothing operator.
The first description, item~\ref{t:raf:i} in Theorem~\ref{t:raf},
shows a pleasing reformulation in terms of the
proximal average. The second description, 
item~\ref{t:raf:ii} in Theorem~\ref{t:raf} is less appealing but has
the advantage
of providing a simple proof of the \emph{self-duality}
\ref{t:raf:iii} observed by Goebel.

\begin{theorem}
\label{t:raf}
Let $f\colon X\to\RX$ be convex, lower semicontinuous and proper,
and let $\lambda \in\zu$.
Then the following hold\,\footnote{Here $\Id\colon X\to X\colon x\mapsto x$
is the \emph{identity operator}.}.
\begin{enumerate}
\item 
\label{t:raf:i}
$\goe{\lambda}{f} = (1+\lambda)\pav{f}{0}{1-\lambda}{\lambda} +
\lambda\qq$.
\item 
\label{t:raf:ii}
$\goe{\lambda}{f} =
(1+\lambda)^2\pav{f}{\qq}{\frac{1-\lambda}{1+\lambda}}{\frac{2\lambda}{1+\lambda}}\circ
(1+\lambda)^{-1}\Id$. 
\item 
\label{t:raf:iii}
\textbf{\emph{(Goebel)}} 
$\big(\goe{\lambda}{f}\big)^* = \goe{\lambda}{(f^*)}$. 
\end{enumerate}
\end{theorem}
\begin{proof}
Let $x\in X$.
Then, using \eqref{e:pavdef} and standard convex calculus, 
we obtain 
\begin{align}
&\quad
\bigg((1+\lambda)^2\pav{f}{\qq}{\frac{1-\lambda}{1+\lambda}}{\frac{2\lambda}{1+\lambda}}\circ
(1+\lambda)^{-1}\Id\bigg)(x)\\[+4mm]
&= (1+\lambda)^2 \bigg(\left(\frac{1-\lambda}{1+\lambda}(f+\qq)^* +
\frac{2\lambda}{1+\lambda}(\qq+\qq)^*\right)^*-\qq\bigg)\left(\frac{x}{1+\lambda}\right)\\[+4mm]
&= (1+\lambda)^2 \left(\frac{1-\lambda}{1+\lambda}\big(f+\qq\big)^* +
\frac{\lambda}{1+\lambda}\qq\right)^*\left(\frac{x}{1+\lambda}\right)
-\qq(x)\\[+4mm]
&= (1+\lambda) \Big((1-\lambda)\big(f+\qq\big)^* +
\lambda\qq\Big)^*\left({x}\right) -\qq(x)\label{e:0326:a}\\[+4mm]
&= (1+\lambda) \bigg(\Big((1-\lambda)\big(f+\qq\big)^* +
\lambda\big(0+\qq\big)^*\Big)^*-\qq\bigg)(x) +\lambda\qq(x)\\[+4mm]
&=\Big((1+\lambda)\pav{f}{0}{1-\lambda}{\lambda}+\lambda\qq\Big)(x).
\end{align}
We have verified that \eqref{e:0326:a} as well as 
the right sides of \ref{t:raf:i} and \ref{t:raf:ii}
coincide. 
Starting from \eqref{e:0326:a} and again applying standard convex caluclus, we see that
\begin{align}
&\quad  (1+\lambda) \Big((1-\lambda)\big(f+\qq\big)^* +
\lambda\qq\Big)^*\left({x}\right) -\qq(x)\\[+4mm]
&=  (1+\lambda) \Big(\Big((1-\lambda)\big(f+\qq\big)^*\Big)^* \infconv
\big(\lambda\qq\big)^*\Big)\left({x}\right) -\qq(x)\\[+4mm]
&=  (1+\lambda) \bigg((1-\lambda)\big(f+\qq\big)\Big(\frac{\cdot}{1-\lambda}\Big) \infconv
\frac{1}{\lambda}\qq\bigg)\left({x}\right) -\qq(x)\\[+4mm]
&=  (1+\lambda)
\inf_{y}\bigg((1-\lambda)\big(f+\qq\big)\Big(\frac{y}{1-\lambda}\Big) +
\frac{1}{\lambda}\qq(x-y)\bigg) -\qq(x)\\[+4mm]
&=  (1+\lambda)
\inf_{y}\bigg((1-\lambda)f\Big(\frac{y}{1-\lambda}\Big) +
(1-\lambda)\qq\Big(\frac{y}{1-\lambda}\Big) +
\frac{1}{\lambda}\qq(x-y)-\frac{1}{1+\lambda}\qq(x)\bigg)\\[+4mm]
&=  (1-\lambda^2)
\inf_{y}\bigg(f\Big(\frac{y}{1-\lambda}\Big) +
\qq\Big(\frac{y}{1-\lambda}\Big) + \frac{1}{\lambda(1-\lambda)}\qq(x-y)
-\frac{1}{1-\lambda^2}\qq(x)\bigg).
\end{align}
Simple algebra shows that for every $y\in X$, 
\begin{equation}
\qq\Big(\frac{y}{1-\lambda}\Big) + 
\frac{1}{\lambda(1-\lambda)}\qq(x-y) -\frac{1}{1-\lambda^2}\qq(x)
= \frac{1}{\lambda}\qq\Big(x-\frac{y}{1-\lambda}\Big) + 
\frac{\lambda}{1-\lambda^2}\qq(x).
\end{equation}
Therefore,
\begin{align}
&\quad  (1+\lambda) \Big((1-\lambda)\big(f+\qq\big)^* +
\lambda\qq\Big)^*\left({x}\right) -\qq(x)\\[+4mm]
&=  (1-\lambda^2)
\inf_{y}\bigg(f\Big(\frac{y}{1-\lambda}\Big) +
\qq\Big(\frac{y}{1-\lambda}\Big) + \frac{1}{\lambda(1-\lambda)}\qq(x-y)
-\frac{1}{1-\lambda^2}\qq(x)\bigg)\\[+4mm]
&=  (1-\lambda^2)
\inf_{y}\bigg(f\Big(\frac{y}{1-\lambda}\Big) +
\frac{1}{\lambda}\qq\Big(x-\frac{y}{1-\lambda}\Big) + 
\frac{\lambda}{1-\lambda^2}\qq(x)
\bigg)\\[+4mm]
&=  (1-\lambda^2)
\inf_{z}\bigg(f(z)+
\frac{1}{\lambda}\qq(x-z)+ 
\frac{\lambda}{1-\lambda^2}q(x)
\bigg)\\[+4mm]
&=  \big((1-\lambda^2)\moro{\lambda}{f} + \lambda\qq\big)(x)\\[+4mm]
&=  \goe{\lambda}{f}(x),
\end{align}
which completes the proof of \ref{t:raf:i} and \ref{t:raf:ii}. 

\ref{t:raf:iii}:
In view of the conjugate formula 
$(\beta^2 h\circ(\beta^{-1}\Id))^* = \beta^2
h^*\circ(\beta^{-1}\Id)$, \ref{t:raf:ii}, and \eqref{e:pavconj}, we obtain
\begin{align}
\big(\goe{\lambda}{f}\big)^* &= \left((1+\lambda)^2\pav{f}{\qq}{\frac{1-\lambda}{1+\lambda}}{\frac{2\lambda}{1+\lambda}}\circ
(1+\lambda)^{-1}\Id\right)^*\\[+4mm]
&=
(1+\lambda)^2\left(\pav{f}{\qq}{\frac{1-\lambda}{1+\lambda}}{\frac{2\lambda}{1+\lambda}}\right)^*\circ
(1+\lambda)^{-1}\Id\\[+4mm]
&=
(1+\lambda)^2\pav{f^*}{\qq^*}{\frac{1-\lambda}{1+\lambda}}{\frac{2\lambda}{1+\lambda}}\circ (1+\lambda)^{-1}\Id\\[+4mm]
&=
(1+\lambda)^2\pav{f^*}{\qq}{\frac{1-\lambda}{1+\lambda}}{\frac{2\lambda}{1+\lambda}}\circ (1+\lambda)^{-1}\Id\\[+4mm]
&= \goe{\lambda}{(f^*)}.
\end{align}
The proof is complete. 
\end{proof}

\begin{remark}
Theorem~\ref{t:raf}\ref{t:raf:i}\&\ref{t:raf:ii} gives
two representations of the Goebel smoothing operator in terms of the
proximal average. Goebel \cite{RafPrivate} discovered a converse formula,
which we state next without proof:
\begin{equation}
\pav{f}{\qq}{\lambda}{1-\lambda} = \frac{(2-\lambda)^2}{4}
\goe{\lambda/(2-\lambda)}{f}\circ \Big(\frac{2}{2-\lambda}\Id\Big). 
\end{equation}
\end{remark}

\begin{example}
\label{ex:rafnorm}
Let $\lambda\in\zu$ and set $f =\|\cdot\|$.
Then, for every $x\in X$,
\begin{equation}
\goe{\lambda}{f}(x) =
\begin{cases}
\displaystyle\frac{\|x\|^2}{2\lambda}, &\text{if
$\|x\|\leq\lambda$;}\\[+4mm]
\displaystyle
\frac{\lambda\|x\|^2}{2} + (1-\lambda^2)\|x\| -
\frac{\lambda(1-\lambda^2)}{2}, &\text{if $\|x\|>\lambda$.}
\end{cases}
\end{equation}
\end{example}
\begin{proof}
Combine \eqref{e:raf} and \eqref{e:moronorm}.
\end{proof}

\section{A new smoothing operator} 

\label{sec:Shawn} 

We now provide a novel smoothing operator that
has a very simple expression in terms of the proximal average.

\begin{definition}[new smoothing operator]
Let $f\colon X\to\RX$ be convex, lower semicontinuous and proper,
and let $\lambda \in\zu$.
Then the $\shawn{\lambda}{f}$ is defined by 
\begin{equation}
\label{e:shawn}
\shawn{\lambda}{f} = \pav{f}{\qq}{1-\lambda}{\lambda}. 
\end{equation}
\end{definition}

\begin{theorem}
\label{t:shawn}
Let $f\colon X\to\RX$ be convex, lower semicontinuous and proper,
and let $\lambda \in\zu$.
Set $\mu=\lambda/(2-\lambda)$. 
Then the following hold.
\begin{enumerate}
\item
\label{t:shawn:i}
$\shawn{\lambda}{f} =
(1-\lambda)\moro{\mu}{f}\circ\Big(\tfrac{2}{2-\lambda}\Id\Big) + \mu\qq$.
\item
\label{t:shawn:ii}
$\big(\shawn{\lambda}{f}\big)^* = \shawn{\lambda}{(f^*)}$. 
\end{enumerate}
\end{theorem}
\begin{proof}
\ref{t:shawn:i}: 
Let $x\in X$.
Then, using \eqref{e:shawn}, \eqref{e:pavdef}
and standard convex calculus, 
we obtain
\begin{align}
\big(\shawn{\lambda}{f}\big)(x) 
&= \big((1-\lambda)(f+\qq)^* +\lambda(\qq+\qq)^*\big)^*(x)-\qq(x)\\[+4mm]
&= \big((1-\lambda)(f+\qq)^*
+\tfrac{\lambda}{2}\qq\big)^*(x)-\qq(x)\\[+4mm]
&= \left((1-\lambda)(f+\qq)\Big(\frac{\cdot}{1-\lambda}\Big) \infconv 
\frac{2}{\lambda} \qq\right)(x)-\qq(x)\\[+4mm]
&= \inf_{y}\left((1-\lambda)f\Big(\frac{y}{1-\lambda}\Big)
+ (1-\lambda)\qq\Big(\frac{y}{1-\lambda}\Big)
+\frac{2}{\lambda}\qq(x-y)-\qq(x)\right)\\[+4mm]
&= (1-\lambda)\inf_{y}\left(f\Big(\frac{y}{1-\lambda}\Big)
+ \qq\Big(\frac{y}{1-\lambda}\Big)
+\frac{2}{\lambda(1-\lambda)}\qq(x-y)-\frac{1}{1-\lambda}\qq(x)\right).
\end{align}
Simple algebra shows that for every $y\in X$, 
\begin{equation}
\qq\big(\frac{y}{1-\lambda}\big)
+\frac{2}{\lambda(1-\lambda)}\qq(x-y)-\frac{1}{1-\lambda}\qq(x)
=
\frac{2-\lambda}{\lambda}q\Big(\frac{2x}{2-\lambda}-\frac{y}{1-\lambda}\Big)+
\frac{\lambda}{(1-\lambda)(2-\lambda)}\qq(x). 
\end{equation}
Therefore,
\begin{align}
\big(\shawn{\lambda}{f}\big)(x)  
&= (1-\lambda)\inf_{y}\left(f\Big(\frac{y}{1-\lambda}\Big)
+ \frac{2-\lambda}{\lambda}q\Big(\frac{2x}{2-\lambda}-\frac{y}{1-\lambda}\Big)+
\frac{\lambda}{(1-\lambda)(2-\lambda)}\qq(x) \right)\\[+4mm]
&= (1-\lambda)\inf_{z}\left(f(z)
+ \frac{2-\lambda}{\lambda}q\Big(\frac{2x}{2-\lambda}-z\Big)\right)+
\frac{\lambda}{2-\lambda}\qq(x) \\[+4mm]
&= (1-\lambda)\Big(f\infconv \frac{1}{\mu}\qq\Big)
\Big(\frac{2x}{2-\lambda}\Big)+ \mu\qq(x),
\end{align}
as claimed.

\ref{t:shawn:ii}:
Using \eqref{e:shawn} and \eqref{e:pavconj}, we get
\begin{equation}
\big(\shawn{\lambda}{f}\big)^* =
\big(\pav{f}{\qq}{1-\lambda}{\lambda}\big)^*
= \pav{f^*}{\qq^*}{1-\lambda}{\lambda}
= \pav{f^*}{\qq}{1-\lambda}{\lambda}
=\shawn{\lambda}{(f^*)}.
\end{equation}
The proof is complete. 
\end{proof}

Note that Theorem~\ref{t:shawn}\ref{t:shawn:i} and 
standard properties of the Moreau envelope 
imply that point-wise
\begin{equation}
\lim_{\lambda\to 0^+} \shawn{\lambda}{f} = f
\end{equation} and
that each $\shawn{\lambda}{f}$ is smooth.

\begin{example}
\label{ex:shawnnorm}
Let $\lambda\in\zu$ and set $f =\|\cdot\|$.
Then, for every $x\in X$,
\begin{equation}
\shawn{\lambda}{f}(x) = 
\begin{cases}
\displaystyle\frac{(2-\lambda)\|x\|^2}{2\lambda}, &\text{if
$\displaystyle\|x\|\leq\frac{\lambda}{2}$;}\\[+4mm]
\displaystyle
\frac{\lambda\|x\|^2}{2(2-\lambda)} 
+ \frac{2(1-\lambda)}{2-\lambda}\|x\| -
\frac{\lambda(1-\lambda)}{2(2-\lambda)}, 
&\text{if $\displaystyle\|x\|>\frac{\lambda}{2}$.}
\end{cases}
\end{equation}
\end{example}
\begin{proof}
Combine \eqref{e:moronorm} and Theorem~\ref{t:shawn}\ref{t:shawn:i}.
\end{proof}

\begin{remark}
Let $f=\|\cdot\|$. 
The explicit formulas provided in
Example~\ref{ex:rafnorm} and Example~\ref{ex:shawnnorm}
imply that 
$\goe{\alpha}{f}\neq\shawn{\beta}{f}$, 
for \emph{all} $\alpha$ and $\beta$ in $\zu$.
Thus, the smoothing operator defined by \eqref{e:shawn} is indeed new
and different from Goebel's smoothing operator.
\end{remark}

\begin{remark}
Given a more complicated function $f$, the explicit computation
of the smoothing operators $\goe{\lambda}{f}$ and $\shawn{\lambda}{f}$ may
not be so easy. However, computational convex analysis provides tools
 \cite{Lucet97,LBT} 
to compute the Moreau envelope numerically which 
--- due to the Moreau envelope formulations 
\eqref{e:raf} and Theorem~\ref{t:shawn}\ref{t:shawn:i} --- 
makes it possible to compute the smoothing operators
$\goe{\lambda}{f}$ and $\shawn{\lambda}{f}$ numerically. 

Finally, other approaches to smooth approximation are:
Ghomi's integral convolution method \cite{Ghomi}, 
Seeger's ball rolling technique \cite{Seeger}, 
and Teboulle's entropic proximal mappings \cite{Teboulle}. 
\end{remark}

\section*{Acknowledgment}

Heinz Bauschke was partially supported by the Natural Sciences and
Engineering Research Council of Canada and
by the Canada Research Chair Program.
Sarah Moffat was partially supported by the Natural
Sciences and Engineering Research Council of Canada.
Xianfu Wang was partially supported by the Natural
Sciences and Engineering Research Council of Canada.

\end{document}